\documentclass[12pt, reqno]{amsart}
\usepackage{amsmath, amsthm, amscd, amsfonts, amssymb, graphicx, color}
\usepackage[bookmarksnumbered, colorlinks, plainpages]{hyperref}
\hypersetup{colorlinks=true,linkcolor=red, anchorcolor=green, citecolor=cyan, urlcolor=red, filecolor=magenta, pdftoolbar=true}
\usepackage{soul}

\newtheorem{theorem}{Theorem}[section]
\newtheorem{lemma}[theorem]{Lemma}
\newtheorem{proposition}[theorem]{Proposition}
\newtheorem{corollary}[theorem]{Corollary}
\theoremstyle{definition}

\theoremstyle{remark}
\newtheorem{remark}[theorem]{Remark}

\def\bh{\mathcal{B}(H)}

\def\cH{H}
\def\wn{w_N}
\def\R{\mathbb{R}}
\def\T{\mathbb{T}}
\numberwithin{equation}{section}

\begin{document}

\setcounter{page}{1}

\title[Generalized numerical radius and related inequalities]{Generalized numerical radius and related inequalities}

\author[T. Bottazzi and C. Conde]{T. Bottazzi$^{1,2}$ and C. Conde$^{3,4}$}

\address{$^1$ Universidad Nacional de R\'io Negro. LaPAC, Sede Andina (8400) S.C. de Bariloche, Argentina.}

\address{$^2$ Consejo Nacional de Investigaciones Cient\'ificas y T\'ecnicas, (1425) Buenos Aires,
	Argentina.}

\address{$^3$Instituto Argentino de Matem\'atica ``Alberto P. Calder\'on", Saavedra 15 3Ú piso, (C1083ACA) Buenos Aires, Argentina}

\address{$^4$Instituto de Ciencias, Universidad Nacional de Gral. Sarmiento, J. M. Gutierrez 1150, (B1613GSX) Los Polvorines, Argentina}

\email{tbottazzi@unrn.edu.ar}
\email{cconde@ungs.edu.ar}


 \begin{abstract}
In \cite{Abu-kittaneh}, Abu Omar and Kittaneh defined a new generalization of the numerical radius. That is, given a norm $N(\cdot)$ on $\bh$, the space of bounded linear operators over a Hilbert space $H$, and $A\in\bh$
$$\wn(A)=\sup\limits_{\theta\in \R}N(Re(e^{i\theta}A)).$$ 
They proved several properties and introduced some inequalities. We continue  with the study of this generalized numerical radius and we develop diverse inequalities involving $w_N$. We also study particular cases when $N(\cdot)$ is  the $p$- Schatten norm with $p>1$.

\end{abstract}
\maketitle

\section{Introduction}

Throughout the paper, $(H, \langle \cdot, \cdot\rangle)$ denotes a separable complex Hilbert space.
Let $\mathcal{B}(H)$ and $\mathcal{K}(H)$ denote the Banach spaces of all bounded operators and all compact operators equipped with the usual operator norm $\|\cdot\|$, respectively. In the case when $\dim(\cH)=n$, we identify $\bh$ with the full matrix algebra  $\mathbb{M}_n$ of all $n\times n$ matrices with entries in the complex field. 
The symbol $I$ stands for the identity operator on $\cH$. 

For $A \in \mathcal{B}(H)$, $A^*$ denotes the adjoint of $A$ and  if $A=A^*$ then $A$ is a selfadjoint operator. We can write
$A = Re(A) + i Im(A),$
in which $Re(A) =\frac{A+A^*}{2}$ and  $Im(A)=\frac{A-A^*}{2i}$ are selfadjoint operators. This is the so called Cartesian decomposition of $A$.

For each operator $A\in \bh$, we consider the numerical range
$
W(A)= \{\langle Ah, h \rangle: \|h\|=1 \}$ and 
$$
w(A)=\sup\{ |\lambda|: \lambda \in W(A)\}, 
$$
 called as numerical radius of $A$.
It is well known that $W(A)$ is a bounded convex subset of the complex plane, which contains the spectrum of $A$ in its closure. The numerical radius $w(\cdot)$ defines a norm on $\bh$ which is equivalent to $\|\cdot\|$. In fact, the following inequalities 
\begin{equation}
\frac 12 \|A\|\leq w(A)\leq \|A\|,
\end{equation}
hold. The first inequality becomes an equality if $A^2=0$ and the second inequality becomes an equality if $A$ is normal. On the  other hand, 
Yamazaki proved in   \cite{yamazaki}  that 
$$
w(A)=\sup_{\theta\in \R}\|Re(e^{i\theta}A)\|.
$$
Recently, motivated by the previous relation, the authors in  \cite{Abu-kittaneh} gave a generalization of the numerical radius, in this way: for any  $A\in \bh$ 
\begin{equation}\label{wndef}
\wn(A)=\sup\limits_{\theta\in \R}N(Re(e^{i\theta}A)),
\end{equation}
where  $N(\cdot)$ is a  norm on $\bh$. We say that $N(\cdot)$ is selfadjoint if $N(A)=N(A^*)$ for any $A\in \bh$ and submultiplicative if $N(AB)\leq N(A)N(B)$ for every $A, B\in \bh.$

We briefly describe the contents of this paper. Section 2 and 3 contain basic definitions, notation and some preliminary results about Schatten ideals, orthogonality and norm-parallelism of bounded linear operators. 
In section 4 we obtain new upper and lower rounds for   $\wn$. Section 5 continues  in a similar way considering $\wn=w_p$ for $p$-Schatten norm with $1\leq p<\infty$ and we focus on characterize  the attainment of the upper bound for  $w_p$.  Finally, in the last section,  we obtain several inequalities involving the norms $N(\cdot)$ and $\wn$ for the product of two operators in $\bh$.

\section{Some facts on $p$-Schatten class}
For an  operator $A \in \bh$, we denote the modulus  by $|A| = (A^*A)^{\frac {1}{2}}$. 
For any compact operator $A\in \mathcal{K}(H)$, let $s_1(A), s_2(A),\cdots $ be the singular values of $A$, i.e.
the eigenvalues of the $|A|$  in decreasing order and repeated
according to multiplicity. For  $p>0$, let
\begin{equation} \label{defipllel}
{\|A\|}_p = \left(\sum_{i = 1}^\infty s_i(A)^p\right)^{\frac{1}{p}} = \left({\rm tr}|A|^p\right)^{\frac{1}{p}},
\end{equation}
where $\rm tr(\cdot)$ is the trace functional, i.e.
\begin{equation}\label{traza}
{\rm tr}(A)=\sum_{j=1}^{\infty} \langle Ae_j,e_j\rangle,
\end{equation}
with  $\{e_j\}_{j=1}^{\infty}$ is an orthonormal basis of $H$. 
Note that this coincides with the usual definition of the trace if $\cH$ is finite-dimensional. We observe that the series \eqref{traza} converges
absolutely and it is independent from the choice of basis.
Equality \eqref{defipllel}
defines a norm (quasi-norm) on the ideal $\mathcal{B}_p(H) = \{A\in \mathcal{K}(H): {\|A\|}_p<\infty\}$ for $1\leq p<\infty$ ($0< p <1$), called the
\textit{$p$-Schatten class}.

The following theorem collects some of the most important properties of $p$-Schatten operators:
\begin{theorem}\label{bp prop}
\begin{enumerate}
\setlength\itemsep{.8em}
\item $\mathcal{B}_p(H) \subseteq \mathcal{K}(H).$
\item $B_{00}(\cH)$, the space of operators of finite rank, is a dense subspace of $\mathcal{B}_p(H).$
\item $\mathcal{B}_p(H)$ is an operator ideal in $\bh$ for $1\leq p <\infty.$
\item  If $p_1 < p_2$ and  $A \in \mathcal{B}_{p_1}(H)$, then $A \in \mathcal{B}_{p_2}(H)$
and $\|A\|_{p_2} \leq \|A\|_{p_1}$.
\item For any $A \in \mathcal{B}_p(H)$ and $T \in \bh$ we have the following inequalities: $$\|A\|\leq \|A\|_p\ ,\ \|A\|_p=\|A^*\|_p\ \text{and}\ \|TA\|_p\leq \|T\| \|A\|_p.$$
\item  For $p>0$,  $A \in  \mathcal{B}_p(H)$  if and only if  $A^*A \in \mathcal{B}_{p/2}(H)$ and in this case $\|A\|_p^2=\|A^*A\|_{p/2}$.
\end{enumerate}
\end{theorem}
 It is known that for $0<p<1$, instead
of the triangle inequality, which does not hold in this case, we have $\|A+B\|_p^p\leq \|A\|_p^p+\|B\|_p^p$ for $A, B \in  \mathcal{B}_p(H)$. The so-called \textit{Hilbert-Schmidt class} $\mathcal{B}_2(H)$ is a Hilbert space under the inner product $\langle A, B \rangle_{HS} := {\rm tr}(B^*A)$.
The ideal $\mathcal{B}_1(H)$ is called the \textit{trace class}. For $1<p<\infty$,  $(\mathcal{B}_p(H), \|.\|_p)$ is a uniformly convex space as consequence of the classical McCarthy-Clarkson inequality (see \cite{mccarthy}, Th. 2.7).

If $x,y \in \cH$, then we denote $x\otimes y$ the rank 1 operator defined on $\cH$ by $(x\otimes y)(z)=\langle z, y\rangle x$, then $\|x\otimes y\|=\|x\|\|y\|=||x\otimes y||_p.$

The usual operator norm and the Schatten $p$-norms are special examples
of unitarily invariant norms, i.e.  that satisfies  the invariance property 
$
||| UXV|||=|||X|||,
$
for any pair of unitary operators $U,V$. On the theory of norm ideals and
their associated unitarily invariant norms, a reference for this subject is \cite{gohberg-krein}.

\section{Orthogonality and norm- parallelism of operators}

Let $(X, \|\cdot\|)$ be a normed space over $\mathbb{K} \in \{\mathbb{R}, \mathbb{C}\}$. We say that $x\in X$ is \textit{norm parallel} to $y\in X$ (\cite{S}), in short $x\parallel y$, if there exists $\lambda\in\mathbb{T}=\{\alpha\in\mathbb{K}: \,\,|\alpha|=1\}$ such that
\begin{equation} \label{defiparallel}
\|x + \lambda y\| = \|x\| + \|y\|.
\end{equation}
In the framework of inner product spaces or uniformly convex spaces, the norm parallel relation is exactly the usual vectorial parallel relation, that is,
$x\parallel y$ if and only if $x$ and $y$ are linearly dependent. In the setting of normed linear spaces, two linearly
dependent vectors are norm parallel, but the converse is false in general. Several characterizations of the norm parallelism for Hilbert space operators
were given in \cite{G, Z, Z.M.1}.

The orthogonality between two vectors of $X$, may be defined in several ways.
The so-called \textit{Birkhoff--James} orthogonality reads as follows (see \cite{B, J}):
for $x,y \in X$ it is said that $x$ is \textit{Birkhoff--James orthogonal} (B-J) to $y$, denoted by $x \perp y$, whenever
\begin{equation} \label{defiBJeq}
\|x\|\leq \|x + \gamma y\|
\end{equation}
for all $\gamma \in \mathbb{K}$. If $X$ is an inner product space, then B-J orthogonality is equivalent to the usual orthogonality given by the inner product.
It is also easy to see that B-J orthogonality is nondegenerate, is homogeneous, but it is neither symmetric nor additive.

There are other definitions of orthogonality with different properties. 
Orthogonality in the setting of Hilbert space operators has attracted attention of several mathematicians.
We cite some papers which studied these notions  in chronological order, see for instance \cite{stampfli, kittaneh_laa_1991, magajna, B.S,  benitez-fernandez-soriano, B.G, P.S.G, sain-paul-hait, BCMWZ}.

Let $\mathcal{B}_p(H)$ be a $p$-Schatten ideal with $p>0$.
According to \eqref{defiparallel}, we say that $A,B\in \mathcal{B}_p(H)$
are norm parallel, denoted by $A{\parallel}^p B$,
if there exists $\lambda\in\mathbb{T}$ such that
${\|A + \lambda B\|}_p = {\|A\|}_p + {\|B\|}_p.$

In \cite{BCMWZ} we characterized the norm- parallelism between two operators in $\mathcal{B}_p(H)$, as follows.

\begin{theorem}\label{th.987}
Let $A, B\in\mathcal{B}_p(H)$ with polar decompositions $A=U|A|$ and $B=V|B|$, respectively. If $1<p<\infty$, then the following conditions are equivalent:
\begin{itemize}
\setlength\itemsep{.8em}
\item[(i)] $A{\parallel}^p B$.
\item[(ii)] ${\|A\|}_p\,\big|{\rm tr}(|A|^{p-1}U^*B)\big| = {\|B\|}_p\,{\rm tr}(|A|^p)$.
\item[(iii)] ${\|B\|}_p\,\big|{\rm tr}(|B|^{p-1}V^*A)\big| = \|A\|_p\,{\rm tr}(|B|^p)$.
\item[(iv)] $A, B$ are linearly dependent.
\end{itemize}
\end{theorem}

Let $A =\begin{bmatrix}
1 & 0 \\
0 & 0
\end{bmatrix}$ and $I = \begin{bmatrix}
1 & 0 \\
0 & 1
\end{bmatrix}$. Then, it is trivial that
$\|A + I\|_1 = 3 = \|A\|_1 + \|I\|_1,$
$\|A + I\| = 2 = \|A\| + \|I\|.$
However, it is evident that $A$ and $I$ are linearly independent.

In the last years, different authors have obtained characterizations of the norm parallelism problem for trace-class operators on a Hilbert space $\cH$.	In the context of an infinite dimensional Hilbert space we refer \cite{li-li}  and \cite{Z} and in  other hand, for a finite dimensional space we mention \cite{Li-Sch}. 

%



\section{Properties of  $w_N$  norms}
We start this section with basic properties of the norm $w_N(\cdot)$.

\begin{proposition}
For any $A\in \bh$ holds $\wn(A)=\max\limits_{\theta\in [0, 2\pi]} N(Re(e^{i\theta}A))$.
\end{proposition}
\begin{proof}
It is a simple consequence of the following facts: the interval $[0, 2\pi]$ is a compact set and the function $\theta \to N(Re(e^{i\theta}A))$ is continuous. 
\end{proof}

\begin{lemma}\label{basic prop}
For any $A, X \in \bh$  the following statements hold.
\begin{enumerate}
	\setlength\itemsep{.8em}
\item $\wn(A)=\sup\limits_{\theta\in \R}N(Im(e^{i\theta}A))$.
\item $\wn(A)=\sup\limits_{\alpha^2+\beta^2=1} N(\alpha Re(A)+\beta Im(A))$.
\item $ \max\{\frac12 N(A), \frac 12 N(A^*)\}\leq \wn(A)\leq \frac12(N(A)+N(A^*))$.
\item If $N(\cdot)$ is submultiplicative then $\wn(AX\pm XA^*)\leq (N(A)+N(A^*))\wn(X)$. 
\item If $\wn(A)=\frac12 N(A)$ then $Re(e^{i\theta}A)\parallel ^N Im(e^{i\theta}A)$ for all $\theta\in \R$.

\end{enumerate}
\end{lemma}
\begin{proof}
Item (4) are proved in \cite{Abu-kittaneh}. The proofs of the rest are straightforward.
\end{proof}

We note that the upper bound  obtained in item (3) is a simple consequence of the following inequality which holds for any norm $N$,  
	$$
	N\left(\frac{e^{i\theta}A+e^{-i\theta}A^*}{2}\right)\leq \frac{N(A)+N(A^*)}{2}.
	$$
	This concept is called midpoint point convexity an it was introduced by Jensen. In the context of continuity, midpoint convexity means convexity, and for this reason we have the following  inequality 
	$$
	N\left(\frac{e^{i\theta}A+e^{-i\theta}A^*}{2}\right)\leq  \int_0^1 N((1-\lambda)e^{i\theta}A+\lambda e^{-i\theta}A^*)\:d\lambda  \leq  \frac{N(A)+N(A^*)}{2}, 
	$$
	as a particular case of the classical Hermite-Hadamard (H-H)  inequality for the function $f_{\theta}(\lambda)=N((1-\lambda)e^{i\theta}A+\lambda e^{-i\theta}A^*)$ where $\theta\in \R$ and $A\in \bh$. For  sake of completeness,  we recall a refinement of  H-H inequality: 
let $f$ be a real valued function which is convex on the interval
		$[a,b]$. Then
		\begin{eqnarray}\label{aa}
		f\left(\frac{a+b}{2}\right)&\leq&
		\frac{1}{b-a} \int_a^bf(t)\,dt
		\leq
		\frac{1}{2}\left[	f\left(\frac{a+b}{2}\right)+\frac{f(a)+f(b)}{2}
		\right]\nonumber\\&\leq& \frac{f(a)+f(b)}{2}.\
		\end{eqnarray}

Inspired by the the previous result, we derive several  inequalities for the generalized numerical radius.  In particular we  obtain an improvement of (3) in Lemma \ref{basic prop}.
	\begin{theorem}
		Let $A\in \bh$ and $N(\cdot)$ be a norm on $\bh$. Then 
	\begin{eqnarray}\label{cota_integral}
	\wn(A)&\leq & \sup\limits_{\theta\in \R} \int_0^1 N((1-\lambda)e^{i\theta}A+\lambda e^{-i\theta}A^*)\:d\lambda  \leq \frac 12 \wn(A) +\frac14(N(A)+N(A^*))\nonumber\\&\leq& \frac12(N(A)+N(A^*)).
	\end{eqnarray}
	\end{theorem}
	\begin{proof}
		For each $\theta \in \R$ we consider the  function $f_{\theta}$ defined above.  It is easy to see that $f_{\theta}$ is a convex function in $[0, 1]$ and so by the previous statement we have
		\begin{eqnarray}
		\frac 12N(e^{i\theta}A+e^{-i\theta}A^*)&\leq& \int_0^1 N((1-\lambda)e^{i\theta}A+\lambda e^{-i\theta}A^*)\:d\lambda \nonumber\\
		&\leq &\frac12\left[ \frac 12N(e^{i\theta}A+e^{-i\theta}A^*) + \frac12(N(A)+N(A^*))\right]\nonumber\\ &\leq& \frac12(N(A)+N(A^*)).\nonumber\ 
		\end{eqnarray}
		Taking the supremum over $\theta \in \R$ we conclude that 
		\begin{eqnarray}
		\wn(A)&\leq & \sup\limits_{\theta\in \R} \int_0^1 N((1-\lambda)e^{i\theta}A+\lambda e^{-i\theta}A^*)\:d\lambda  \leq \frac 12 \wn(A) +\frac14(N(A)+N(A^*))\nonumber\\&\leq& \frac12(N(A)+N(A^*)).\nonumber\
		\end{eqnarray}
This completes the proof of the theorem.
	\end{proof}

\begin{proposition}
Let $\phi:[0, \infty)\to \R$ be any nondecreasing convex function or midpoint convex function and $N(\cdot)$ be a norm on $\bh$,  then
\begin{equation}\label{Dragomirradius2}
\phi(\wn(A))\leq \phi\left(\frac12 \left[\wn(A) +\frac12(N(A)+N(A^*)\right]\right)\leq \frac12\phi(\wn(A))+\frac12\phi\left(\frac12(N(A)+N(A^*))\right).\nonumber\\
\end{equation}
\end{proposition}
Clearly convexity implies midpoint-convexity. However, there exist midpoint-convex functions that are not convex. Such functions can be very strange and interesting. In particular if $\phi(x)=x^r$ with $r\geq 1$ we have
$$
\wn^r(A)\leq \frac{1}{2^r} \left[\wn(A) +\frac12(N(A)+N(A^*))\right]^r\leq \frac12\wn^r(A)+\frac {1}{2^{r+1}}\left(N(A)+N(A^*)\right)^r.
$$

\begin{remark}Let $N(\cdot)$ be a norm on $\bh$. Following  Kikianty and Dragomir in \cite{KD} we introduce a norm on $\bh\times \bh$, 
	which we call the $r-HH$-norm induced by $N(\cdot)$,  and is defined as follows:
$$
N((A, B))_{r-HH}:=\sup\limits_{\theta\in \R}\left( \int_0^1 N((1-\lambda)e^{i\theta}A+\lambda e^{-i\theta}B^*)^r\:d\lambda\right)^{\frac 1r},
$$
for any $1\leq r<\infty$ and $(A, B)\in \bh\times \bh$. From the classical H-H inequality, we obtain 
$$
\sup\limits_{\theta\in \R} N\left(\frac{e^{i\theta}A+e^{-i\theta}B^*}{2}\right)\leq \sup\limits_{\theta\in \R}\left( \int_0^1 N((1-\lambda)e^{i\theta}A+\lambda e^{-i\theta}B^*)^r\:d\lambda\right)^{\frac 1r}\leq \left(\frac{N(A)^r+N(B)^r}{2}\right)^{\frac 1r}, 
$$
and in particular if we consider the pair $(A, A^*)$ we get a norm on $\bh$ given by $\overline{N}(A)_{r-HH}:=N((A, A^* ))_{r-HH}$ and 
\begin{eqnarray}\label{wvsHH}
\wn(A)\leq \sup\limits_{\theta\in \R}\left( \int_0^1 N((1-\lambda)e^{i\theta}A+\lambda e^{-i\theta}A^*)^r\:d\lambda\right)^{\frac 1r} \leq \left(\frac{N(A)^r+N(A^*)^r}{2}\right)^{\frac 1r}.
\end{eqnarray}
It is clear that $\overline{N}(\cdot)_{r-HH}$ is a selfadjoint norm on $\bh$ and $\overline{N}(A)_{r-HH}=N(A)$ if $A$ is a selfadjoint operator.
	We note that $\overline{N}(\cdot)_{1-HH}$ and $\wn(\cdot)$ norms are equivalent in $\bh$, since by inequality \eqref{cota_integral} we have 
	$$
	\wn(A)\leq \sup\limits_{\theta\in \R}\int_0^1 N((1-\lambda)e^{i\theta}A+\lambda e^{-i\theta}A^*)\:d\lambda \leq \frac 32 \wn(A).
	$$
\end{remark}

\begin{proposition}
	Let $A\in \bh$ and $r\geq 1$.  Then $w_{\overline{N}_{r-HH}}(A)=\wn(A).$
	\begin{proof}
		For each $\theta \in \R$ the operator $Re(e^{i\theta}A)$ is selfadjoint, then
		$$
		w_{\overline{N}_{r-HH}}(A)= \sup\limits_{\theta\in \R} {\overline{N}_{r-HH}}(Re(e^{i\theta}A))=\sup\limits_{\theta\in \R} N(Re(e^{i\theta}A))=\wn(A).
		$$
	\end{proof}
	\end{proposition}

In the next theorem, we establish a lower bound for the generalized numerical radius.

\begin{theorem}
	Let $A\in \bh$ and $N(\cdot)$ be a norm on $\bh$. Then 
		\begin{eqnarray}\label{desigualdad2}
(0\leq)2 \sup\limits_{\theta\in \R} \int_0^1 N((1-\lambda)e^{i\theta}A+\lambda e^{-i\theta}A^*)\:d\lambda  -\frac12(N(A)+N(A^*))\leq w_N(A).
	\end{eqnarray}
\begin{proof}
We have by inequality \eqref{cota_integral}
\begin{eqnarray}
0&\leq& \max\{\frac12 N(A), \frac 12 N(A^*)\}-\frac 14 (N(A)+N(A^*))\leq w_N(A)-\frac 14 (N(A)+N(A^*))\nonumber\\
&\leq & \sup\limits_{\theta\in \R} \int_0^1 N((1-\lambda)e^{i\theta}A+\lambda e^{-i\theta}A^*)\:d\lambda -\frac14(N(A)+N(A^*)) \leq \frac 12 \wn(A). \nonumber\
\end{eqnarray}
\end{proof}
\end{theorem}

	\begin{corollary}
		Let $A\in \bh$ and $N(\cdot)$ be a norm on $\bh$. Then 
		\begin{eqnarray}\label{desigualdad1}
		\frac14(N(A)+N(A^*))\leq  \sup\limits_{\theta\in \R} \int_0^1 N((1-\lambda)e^{i\theta}A+\lambda e^{-i\theta}A^*)\:d\lambda \leq 	\frac12(N(A)+N(A^*)) .
		\end{eqnarray}
	\end{corollary}
	\begin{proof}
		It is an immediate consequence of \eqref{cota_integral} and \eqref{desigualdad2}.
\end{proof}

Our next goal  is to determine when $\wn$ coincides with the upper bound $\frac 1 2(N(A)+N(A^*))$.
In the context of bounded linear operators on Hilbert spaces, we notice that if $A\in \mathcal{B}(H)$ satisfies $\|A\|=w(A)$ (i.e. $A$ is a normaloid operator) if and only if  $A\parallel I$ (see Proposition 4.2 in \cite{BCMWZ}).  In \cite{Z.M.1}, the authors investigated the case when an operator is parallel to the identity operator in the context of $\bh$. Combining the previous results we have the following characterization: let $A\in \bh$ then 
$$
w(A)=\|A\|  \Leftrightarrow A\parallel I \Leftrightarrow  A\parallel A^*.
$$ 
We obtain an analogous result for $\wn$.
\begin{theorem}\label{cotasuperior}
Let $A\in \bh$. The following conditions are equivalent:
\begin{enumerate}
	\setlength\itemsep{.5em}
\item $\wn(A)=\frac12(N(A)+N(A^*)).$
\item $A{\parallel}^N A^*$.
\end{enumerate}
Moreover, if $N(\cdot)$ is selfadjoint the previous conditions are equivalent to $\wn(A)=N(A)$.
\end{theorem}
\begin{proof}
Suppose that $A{\parallel}^N A^*$, then there exists $\lambda=e^{i2\theta_0}\in \T$ such that $N(A+\lambda A^*)=N(A)+N(A^*).$ This equality implies 
$$
N\left( \frac{e^{-i\theta_0}A+e^{i\theta_0}A^*}{2}\right)=\frac12(N(A)+N(A^*)).
$$ 
On the other hand, if $\wn(A)=\frac12(N(A)+N(A^*))$ there exists $\theta_0\in [0, 2\pi]$ such that 
$N(Re(e^{i\theta_0}A))=\frac12(N(A)+N(A^*))$, that is $A{\parallel}^N A^*$.
\end{proof}

\begin{corollary}
Let $A\in \bh$ and $N(\cdot)$ a selfadjoint and strictly convex norm on $\bh$. Then, the following conditions are equivalent:
\begin{enumerate}
	\setlength\itemsep{.5em}
\item $\wn(A)=N(A)$.
\item $A{\parallel}^N A^*$.
\item $A=\alpha A^*$ with $\alpha \in \T$.
\end{enumerate}
\end{corollary}


\section{The case $w_p$}


In this section we focus on the particular case $N(\cdot)=\|\cdot\|_p$ for $1\leq p<\infty$, that is 
	$$w_p(A)=\sup_{\theta\in \R}\|Re(e^{i\theta}A)\|_p,$$
	for any $A\in \mathcal{B}_{p}(H)$.

First we will focus on remembering and obtaining new bounds for $w_2(\cdot)$,  for the particular case of  the Hilbert-Schmidt norm. In \cite{Abu-kittaneh}, the authors proved an explicit formula of 
  $w_2(\cdot)$    in terms of $\|A\|_2$ and $tr(A^2)$, more precisely
$$
w_2(A)=\sqrt{\frac{\|A\|_2^2+|tr(A^2)|}{2}}
$$
for any $A\in \mathcal{B}_2(\cH)$ and  the existence of lower and upper bounds, which are
$$\frac{1}{\sqrt{2}}\|A\|_2\leq w_2(A)\leq \|A\|_2. $$

The following result improve the existence of a new lower bound for $w_2(\cdot)$.

\begin{theorem}
	Let $A\in \mathcal{B}_2(\cH)$. Then
	\begin{eqnarray}
	\max \left\lbrace  \frac{1}{\sqrt 2} \|A\|_2, \frac{\|A\|_2+|tr(A^2)|^{1/2}}{2} \right\rbrace \leq w_2(A).
	\end{eqnarray}
	
\end{theorem}
\begin{proof}
	For  $a, b\geq 0$, the arithmetic mean $\mathcal{A}(a, b)$ and  quadratic mean $\mathcal{Q}(a, b)$ are, respectively,
	defined by 
	$$
	\mathcal{A}(a, b)=\frac{a+b}{2} \qquad {\rm and} \qquad \mathcal{Q}(a, b)=\sqrt{\frac{a^2+b^2}{2}}.
	$$
	It is well known that $\mathcal{A}(a, b)\leq \mathcal{Q}(a, b)$, and this inequality implies that 
	\begin{eqnarray}
	\frac{\|A\|_2+|tr(A^2)|^{1/2}}{2}=\mathcal{A}(\|A\|_2, |tr(A^2)|^{1/2})\leq \mathcal{Q}(\|A\|_2, |tr(A^2)|^{1/2})=w_2(A). 
	\end{eqnarray}
\end{proof}

Observe that $\frac{1}{\sqrt 2} \|A\|_2$ is not comparable to $\dfrac{\|A\|_2+|tr(A^2)|^{1/2}}{2}$, as we see in the next examples:
\begin{enumerate}
	\item If $A^2=0$ but $A\neq 0$, then $\frac{1}{\sqrt 2} \|A\|_2>\frac{\|A\|_2}{2}$.
	\item Let $A=\begin{pmatrix}
	\alpha&0\\
	0&i\beta
	\end{pmatrix}$, with $\alpha,\beta\in \R$. Then, $A^2=\begin{pmatrix}
	\alpha^2&0\\
	0&-\beta^2
	\end{pmatrix}$ and 
	$$\|A\|_2+|tr(A^2)|^{1/2}=\sqrt{\alpha^2+\beta^2}+\sqrt{\left|\alpha^2-\beta^2\right|}.$$
	In particular, if $\alpha=1$, $\beta=1+\frac 1 n$ and $n=\frac{1}{10}$, 
	$$\left. \begin{array}{l}
	\|A\|_2+|tr(A^2)|^{1/2}=\dfrac{\sqrt{122}+\sqrt{120}}{2}\approx 10.999\\ 
	\frac{1}{\sqrt 2}\|A\|_2=\frac{\sqrt{122}}{\sqrt{2}}\approx 7.81.
	\end{array}\right\rbrace \Rightarrow \|A\|_2+|tr(A^2)|^{1/2}>\frac{1}{\sqrt 2}\|A\|_2$$
	On the other hand, if  $n\in \mathbb{R}$ is chosen that $\left|\frac 2 n+\frac{1}{n^2} \right|<0.001$, the reverse inequality holds.
	\item Moreover, if any $A$ fulfills $|tr(A^2)|^{1/2}>(\sqrt{2}-1)\|A\|_2$, then
	$$\sqrt{2}\|A\|_2<|tr(A^2)|^{1/2}+\|A\|_2\ \Rightarrow \frac{1}{\sqrt{2}}\|A\|_2<\dfrac{|tr(A^2)|^{1/2}+\|A\|_2}{2}.$$
\end{enumerate}

\begin{remark}
	Observe that if $A\neq 0$, 
	$$w_2(A)=\dfrac{|tr(A^2)|^{1/2}+\|A\|_2}{2}\Leftrightarrow|tr(A^2)|^{1/2}=\|A\|_2,$$ which means that 
	$$\dfrac{|tr(A^2)|^{1/2}+\|A\|_2}{2}=\|A\|_2=w_2(A),$$
	which occurs if and only if $A$ is normal and the squares of its nonzero eigenvalues have the same argument (by Corollary 2 in \cite{Abu-kittaneh}). 
\end{remark}

\begin{remark}
Using \eqref{wvsHH} for a nonzero operator $A\in \mathcal{B}_2(\cH)$ we have
	\begin{eqnarray}
	\frac{\|A\|_2^2+|tr(A^2)|}{2}\leq\sup\limits_{\theta\in \R}  \int_0^1\| ((1-\lambda)e^{i\theta}A+\lambda e^{-i\theta}A^*\|_2^2\:d\lambda \leq \|A\|_2^2.
	\end{eqnarray}
	but in this case the integral expression has an explicit formula, since
	\begin{eqnarray*}
		\int_0^1\| ((1-\lambda)e^{i\theta}A+\lambda e^{-i\theta}A^*\|_2^2\:d\lambda&=&\int_0^ 1\left[ ((1-\lambda)^2+\lambda^2)\|A\|_2^2+2\lambda(1-\lambda)tr(Re(e^{2i\theta}A^2))\right] d\lambda
	\end{eqnarray*}
	\begin{eqnarray*}
		&=&\|A\|_2^2\int_0^1\left(2\lambda^2-2\lambda+1\right)d\lambda +2 tr(Re(e^{2i\theta}A^2))\int_0^1\lambda-\lambda^2 d\lambda\\
		&=&\frac 2 3\|A\|_2^2+\frac 1 3tr(Re(e^{2i\theta}A^2))=\frac 2 3 \|A\|_2^2+\frac 1 3\left(2 \|Re(e^{i\theta}A)\|_2^2-\|A\|_2^2\right)\\
		&=&\frac 1 3\|A\|_2^2+\frac 2 3\|Re(e^{i\theta}A)\|_2^2
	\end{eqnarray*}
	holds for any $\theta\in \R$. Then,
	$$\frac{\|A\|_2^2+3|tr(A^2)|}{4}\leq\|Re(e^{i\theta}A)\|_2^2\leq \|A\|_2^2$$
	and therefore
	\begin{eqnarray}\label{bounds for w2 new}
	\frac{\sqrt{ \|A\|_2^2+3|tr(A^2)|}}{2} \leq w_2(A)\leq \|A\|_2.
	\end{eqnarray}	
The lower bound in \eqref{bounds for w2 new} is new but it is not better than $\frac{1}{\sqrt{2}}\|A\|_2$, since for $a,b\geq 0$
$\frac{1}{\sqrt{2}}a\leq \frac{\sqrt{a^2+3b^2}}{2}$ implies $a\leq b.$
\end{remark}

%
%

In finite dimension, the attainment of the lower bound can be related with Birkhoff--James orthogonality, as it is shown in the following statement.
\begin{proposition}\label{equiv cota inf}
	Let $A\in \mathbb{M}_n$, then the following conditions are equivalent
	\begin{enumerate}
		\setlength\itemsep{.5em}
		\item $w_2(A)=\frac{1}{\sqrt{2}}\|A\|_2.$
		\item $I\perp^p A^2$ for $1\leq p<\infty$.
	\end{enumerate}
Moreover, if {\rm(1)} and {\rm(2)} are valid, then $ I  \perp^{|||\cdot|||}A^2$ for  any unitarily invariant norm $|||\cdot|||$.
\end{proposition}



Now, we focus in the general case when $p\in [1, \infty)$ and $p\neq 2.$ From Theorem 2 in \cite{Abu-kittaneh} follows that 
\begin{eqnarray}\label{cotabasica}
\frac12\|A\|_p\leq w_p(A)\leq \|A\|_p
\end{eqnarray}
 for any $1\leq p<\infty$ and $A\in \mathcal{B}_p(H)$.  Recall that if 
 	$1\leq p_1< p_2$ by (4) in Theorem \ref{bp prop}, this implies that 
 \begin{eqnarray}\label{cotap1p2}
w_{p_2}(T)\leq w_{p_1}(T).
\end{eqnarray}
 

We develop a number of inequalities to obtain new bounds
for the generalized numerical radius $w_p$  using the properties of the $p$-Schatten norms.

	\begin{proposition}\label{cotas1}
	Let $A\in\mathcal{B}_p(H)$, then 
	\begin{eqnarray}\label{p-equiv}
	2^{- \frac 1p}\|A\|_p\leq w_p(A) \leq \|A\|_p 
	\end{eqnarray}
	for $1\leq p\leq 2,$ and 
	\begin{eqnarray}
	2^{\frac 1p -1}\|A\|_p\leq w_p(A) \leq \|A\|_p
	\end{eqnarray}
 for $2\leq p<\infty.$
	\end{proposition}
	\begin{proof}
	We only prove the left hand side of \eqref{p-equiv}. By  Theorem 1 in \cite{BK}  we have that for any $\theta \in \R$ hold 
	\begin{eqnarray}
	2^p \|A\|_p\leq 2^p(\|Re(e^{i \theta}A)\|_p^p+\|Im(e^{i \theta}A)\|_p^p)\leq 2^{p+1}w_p^p(A).\nonumber
\end{eqnarray}
	\end{proof}
Observe that in both cases, the lower bounds found improve $\frac 12\|A\|_p$. Also, note that if $p=2$, we obtain Theorem 8 in \cite{Abu-kittaneh}.

\begin{proposition}
	Let $A\in\mathcal{B}_p(H)$,  with $1<p<\infty$,  such that $w_p(A)=\|A\|_p$ then $\|A^2\|_{p/2}=\|A\|^2_p. $
\end{proposition}

\begin{proof}
	It is a direct consequence of Corollary \ref{p-normaloid} and Theorem \ref{bp prop}. 
\end{proof}

Let $A\in\mathcal{B}_p(H)$ with $p\geq 1$,  then for any $\theta \in \R$ we get $Im^2(e^{i\theta} A)=\frac12( A^*A+AA^*)-Re^2(e^{i\theta} A)\geq 0.$ It follows by Weyl's monotonicity principle that if $S, T \in  \mathcal{K}(H)$ are positive and $S\leq T$, then $s_j(S)\leq s_j(T)$ for any $j\in \mathbb{N}$, so $\|Re^2(e^{i\theta} A)\|_{p/2}\leq \frac 12\| A^*A+AA^*\|_{p/2}.$ This shows that
	\begin{eqnarray}\label{upperkittaneh}  w_p^2(A)\leq  \frac 12 \| A^*A+AA^*\|_{p/2}
	\end{eqnarray}

To see that inequality \eqref{upperkittaneh} improves inequality \eqref{cotabasica} for $p\geq 2$, consider the chain of inequalities 
	$$
	w_{p}^2(A)\leq \frac12 \|A^*A+AA^*\|_{p/2}\leq \frac12\|A^*A\|_{p/2}+ \frac12 \|AA^*\|_{p/2}\leq \|A\|_p^2.
	$$

In the next theorem, we give a lower and upper bound for the generalized numerical radius $w_p$ for $p\geq 2$.

\begin{theorem}\label{cotas}
Let $A\in\mathcal{B}_p(H)$ with $p\geq 2$, then 
\begin{eqnarray}\label{upper-lower} \sup\limits_{\theta\in \R} \frac{\left|\: \|Re^2(e^{i\theta}A)\|_{p/2}- \| Im^2(e^{i\theta}A)\|_{p/2}\right|}{2}\leq  w_p^2(A)-  \frac 14 \| A^*A+AA^*\|_{p/2}\leq \frac12 w_{p/2}(A^2)
\end{eqnarray}
\end{theorem}
\begin{proof}
First, observe that mimicking the idea in \cite{ZMXF}, Theorem 2.3 we obtain for any $\theta\in \mathbb{R}$
\begin{eqnarray}
w_p^2(A)&\geq& \max\{\|Re(e^{i\theta}A)\|_p^2, \| Im^2(e^{i\theta}A)\|_p^2\}=\max\{\|Re^2(e^{i\theta}A)\|_{p/2}, \| Im^2(e^{i\theta}A)\|_{p/2}\}\nonumber\\
&=&\frac{\|Re^2(e^{i\theta}A)\|_{p/2}+\|Im^2(e^{i\theta}A)\|_{p/2}
}{2}+\frac{\left|\: \|Re^2(e^{i\theta}A)\|_{p/2}- \| Im^2(e^{i\theta}A)\|_{p/2}\right|}{2}\nonumber\\
&\geq&\frac 12\| Re^2(e^{i\theta}A)+Im^2(e^{i\theta}A)\|_{p/2} +\frac{\left|\: \|Re^2(e^{i\theta}A)\|_{p/2}- \| Im^2(e^{i\theta}A)\|_{p/2}\right|}{2}\nonumber\\ 
&=&\frac 14\| A^*A+AA^*\|_{p/2} +\frac{\left|\: \|Re^2(e^{i\theta}A)\|_{p/2}- \| Im^2(e^{i\theta}A)\|_{p/2}\right|}{2}. \nonumber\
\end{eqnarray}
Hence
\begin{eqnarray}
\frac 14\| A^*A+AA^*\|_{p/2} + \sup\limits_{\theta\in \R} \frac{\left|\: \|Re^2(e^{i\theta}A)\|_{p/2}- \| Im^2(e^{i\theta}A)\|_{p/2}\right|}{2}\leq w_p^2(A).
\end{eqnarray}
To prove the second inequality in \eqref{upper-lower}, we have 
\begin{eqnarray}\label{upperbound}
w_p^2(A)&=&\sup\limits_{\theta\in \R}\|Re(e^{i\theta}A)\|_p^2=\sup\limits_{\theta\in \R}\frac14 \|e^{i\theta}A+e^{-i\theta}A^*\|_p^2=\sup\limits_{\theta\in \R}\frac14 \|(e^{i\theta}A+e^{-i\theta}A^*)^2\|_{p/2}\nonumber\\
&=&\sup\limits_{\theta\in \R}\frac14 \| A^*A+AA^*+2Re(e^{2i\theta}A^2)\|_{p/2}\nonumber\\
&\leq& \frac 14 \| A^*A+AA^*\|_{p/2} +\frac 12\sup\limits_{\theta\in \R} \|Re(e^{2i\theta}A^2)\|_{p/2}\nonumber\\
&=& \frac 14 \| A^*A+AA^*\|_{p/2} +\frac12 w_{p/2}(A^2)
\end{eqnarray}
\end{proof}

\begin{remark}\begin{enumerate}
		\item 
The inequalities in Theorem \ref{cotas}
can be extended to $Q$-norms. For
more examples of these norms, the reader is referred to \cite{BK90}.

	\item  To see that left inequality in \eqref{upper-lower} improves Theorem 2.3 in  \cite{ZMXF},  for $p$-Schatten  norms with $p\geq 2$ we consider property \eqref{cotap1p2}.

\item Let $A\in\mathcal{B}_p(H)$ with $p\geq 2$. Considering the polar decomposition of $A$ and Theorem 1  in \cite{BK902} we have 
	\begin{eqnarray}
	s_j(A^2)=s_j(U|A|U|A|)\leq  \frac 12 s_j(A^*A+AA^*).
	\end{eqnarray}
	Thus
\begin{eqnarray}
w_{p/2}(A^2)&\leq& \|A^2\|_{p/2}=\left(\sum _{j=1}^{\infty}s_j(A^2)^{p/2}\right)^{2/p}\leq \frac 12 \left(\sum _{j=1}^{\infty}s_j(A^*A+AA^*)^{p/2}\right)^{2/p}\nonumber\\&\leq&
\frac 12 \|A^*A+AA^*\|_{p/2},
\end{eqnarray} 
and hence Theorem \ref{cotas} refines inequality \eqref{upperkittaneh}.
	\end{enumerate}
\end{remark}
For $0<p<1$, instead
	of the triangle inequality, which does not hold for the two-sided ideal $ \mathcal{B}_p(H)$, we have $\|A+B\|_p^p\leq \|A\|_p^p+\|B\|_p^p$ for $A, B\in  \mathcal{B}_p(H)$.
Utilizing a similar argument as in Theorem \ref{cotas} we get the following statement.
\begin{proposition}
	Let $A\in\mathcal{B}_p(H)$ with $1\leq p< 2$, then 
	\begin{eqnarray}\label{upper-lower-2}
	\frac {1}{2^{p/2+1}}\| A^*A+AA^*\|_{p/2}^{p/2}&+& \sup\limits_{\theta\in \R} \frac{\left|\: \|Re^2(e^{i\theta}A)\|_{p/2}^{p/2}- \| Im^2(e^{i\theta}A)\|_{p/2}^{p/2}\right|}{2}\leq w_p^p(A)\nonumber\\
	&\leq&\frac{1}{2^p} \| A^*A+AA^*\|_{p/2}^{p/2} +\frac {1}{2^{p/2}} w_{p/2}^{p/2}(A^2).
	\end{eqnarray}
	\end{proposition}
	\begin{proof}
First, observe that mimicking the idea in \cite{ZMXF}, Theorem 2.3 we obtain for any $\theta\in \mathbb{R}$
\begin{eqnarray}
w_p^p(A)&\geq& \max\{\|Re(e^{i\theta}A)\|_p^p, \| Im(e^{i\theta}A)\|_p^p\}=\max\{\|Re^2(e^{i\theta}A)\|_{p/2}^{p/2}, \| Im^2(e^{i\theta}A)\|_{p/2}^{p/2}\}\nonumber\\
&=&\frac{\|Re^2(e^{i\theta}A)\|_{p/2}^{p/2}+\|Im^2(e^{i\theta}A)\|_{p/2}^{p/2}
}{2}+\frac{\left|\: \|Re^2(e^{i\theta}A)\|_{p/2}^{p/2}- \| Im^2(e^{i\theta}A)\|_{p/2}^{p/2}\right|}{2}\nonumber\\
&\geq&\frac 12\| Re^2(e^{i\theta}A)+Im^2(e^{i\theta}A)\|_{p/2} ^{p/2}+\frac{\left|\: \|Re^2(e^{i\theta}A)\|_{p/2}^{p/2}- \| Im^2(e^{i\theta}A)\|_{p/2}^{p/2}\right|}{2}\nonumber\\ 
&=&\frac {1}{2^{p/2+1}}\| A^*A+AA^*\|_{p/2}^{p/2} +\frac{\left|\: \|Re^2(e^{i\theta}A)\|_{p/2}^{p/2}- \| Im^2(e^{i\theta}A)\|_{p/2}^{p/2}\right|}{2}. \nonumber\
\end{eqnarray}
Hence
\begin{eqnarray}
\frac {1}{2^{p/2+1}}\| A^*A+AA^*\|_{p/2}^{p/2}+ \sup\limits_{\theta\in \R} \frac{\left|\: \|Re^2(e^{i\theta}A)\|_{p/2}^{p/2}- \| Im^2(e^{i\theta}A)\|_{p/2}^{p/2}\right|}{2}\leq w_p^p(A).
\end{eqnarray}
To prove the second inequality in \eqref{upper-lower-2}, we have 
\begin{eqnarray*}
w_p^p(A)&=&\sup\limits_{\theta\in \R}\|Re(e^{i\theta}A)\|_p^p=\sup\limits_{\theta\in \R}\frac{1}{2^p} \|e^{i\theta}A+e^{-i\theta}A^*\|_p^p=\sup\limits_{\theta\in \R}\frac{1}{2^p}  \|(e^{i\theta}A+e^{-i\theta}A^*)^2\|_{p/2}^{p/2}\nonumber\\
&=&\sup\limits_{\theta\in \R}\frac{1}{2^p} \| A^*A+AA^*+2Re(e^{2i\theta}A^2)\|_{p/2}^{p/2}\nonumber\\
&\leq&\frac{1}{2^p} \| A^*A+AA^*\|_{p/2} ^{p/2}+\frac {1}{2^{p/2}}\sup\limits_{\theta\in \R} \|Re(e^{2i\theta}A^2)\|_{p/2}^{p/2}\nonumber\\
&=&\frac{1}{2^p} \| A^*A+AA^*\|_{p/2}^{p/2} +\frac {1}{2^{p/2}} w_{p/2}^{p/2}(A^2)
\end{eqnarray*}
\end{proof}

Since $(\mathcal{B}_p(H), \|.\|_p)$ is a uniformly convex space for $1<p<\infty$, we used the characterization of the norm-parallelism in $p$-Schatten ideals  obtained in \cite{BCMWZ} (using the notion of semi-inner product in the sense of Lumer) and Theorem \ref{cotasuperior} to obtain the following statement  for  $w_p$. 

\begin{proposition}\label{p-normaloid}
	Let $A\in\mathcal{B}_p(H)$ with polar decompositions $A=U|A|$ and $A^*=V|A^*|$, respectively. If $1<p<\infty$, then the following conditions are equivalent:
	\begin{enumerate}
		\setlength\itemsep{.5em}
		\item $w_p(A)=\|A\|_p$.
		\item$A{\parallel}^p A^*$.
		\item $\big|{\rm tr}(|A|^{p-1}U^*A^*)\big| = \|A\|_p^p$.
		\item $\big|{\rm tr}(|A^*|^{p-1}V^*A)\big| = \|A^*\|_p^p$.
		\item$A=\alpha A^*$ with $\alpha \in \T$.
	\end{enumerate}
\end{proposition}

As a consequence of the preceding results and a combination of the different characterizations of norm parallelism for trace-class operators (in a context of a finite or infinite dimensional space) we obtain the following statement. 
\begin{corollary}
	Let $A\in\mathcal{B}_1(H)$,  then the following conditions are equivalent:
	
	\begin{enumerate}
		\setlength\itemsep{.5em}
		\item $w_1(A)=\|A\|_1$.
		\item There exists $\theta \in \R$ such that $\|A+e^{i\theta}A^*\|_1=\|A\|_1+\|A^*\|_1$  (i.e. $A{\parallel}^1 A^*$).
		
		\item There exist a partial isometry $V$ and $\lambda\in\mathbb{T}$ such that $A = V|A|$ and $A^*= \lambda V|A^*|.$
		\item There exists $\lambda\in\mathbb{T}$ such that
		$$\Big|{\rm tr}(|A|) + \lambda\, {\rm tr}(U^*A^*)\Big|\leq\Big{\|P_{\ker A^*}(A + \lambda A^*)P_{\ker A}\Big\|}_1,$$
		where $A = U|A|$ is the polar decomposition of $A$. 
		\item  $(A^*)^2=\lambda |A| |A^*|$
		\item There exist isometries $V, W$ and $\lambda\in\mathbb{T}$  such that 
		$$|A+\lambda A^*|=V|A|V^*+W|A^*|W^*.$$
		\item There exists $\lambda\in\mathbb{T}$ such that $| A+\lambda A^*|=|A|+|A^*|.$ (i.e. absolute value parallel definition in \cite{Z}).
		\item There exist a closed subspace $\mathcal{M}$ of $\cH$ such that  $$\|A\|_1=tr(P_{\mathcal{M}}e^{-i\frac{\theta}{2}}AP_{\mathcal{M}})-tr(P_{\mathcal{M}^{\perp}}e^{-i\frac{\theta}{2}}AP_{\mathcal{M}^{\perp}}),$$
		with $P_{\mathcal{M}}e^{-i\frac{\theta}{2}}AP_{\mathcal{M}}\geq 0$ and $P_{\mathcal{M}^{\perp}}e^{-i\frac{\theta}{2}}AP_{\mathcal{M}^{\perp}}\leq 0\ $.
	\end{enumerate}
	If $\cH$ is finite dimensional then (1) to (8) are also equivalent to
	\begin{enumerate}
		\setcounter{enumi}{8}
		\item There exits $F\in \mathbb{M}_n$ such that $\|F\|\leq 1$, $tr(AF^*)=\|A\|_1=|tr(A^*F^*)|.$
	\end{enumerate}
	Furthermore if $A$ is invertible, all the previous conditions are equivalent to 
	\begin{enumerate}
		\setcounter{enumi}{9}
		\item$\Big|{\rm tr}\big(|A|A^{-1}A^*\big)\Big| = {\|A^*\|}_1$.
	\end{enumerate}
	\begin{proof}
		It is a direct consequence of Theorem \ref{cotasuperior}, Theorem 2.15 and Corollary 2.16 in \cite{Z}, Theorem 2.3 and 2.6 in \cite{li-li} and Theorem 3.3 in \cite{Li-Sch}.
	\end{proof}
\end{corollary}

\section{Inequalities involving $w_N$ for products of operators}

The following general result for the product of two operators holds, as a simple consequence of the upper bound of $\wn$ with $N(\cdot)$ a selfadjoint norm on $\bh$: if $A, X\in \bh$ then 
\begin{equation}\label{ineq_producto}
\wn(AX)\leq N(AX)\leq N(A)N(X)\leq 4\wn(A)\wn(X).\end{equation}
New estimates for the generalized numerical radius, $\wn$ with $N(\cdot)$ a selfadjoint norm,  of the product $AX$ are given in this section. 
We begin with the following definitions which are necessaries. Let $A, T\in \bh$, the vector-function $A-\lambda T$ is known as the pencil generated by $A$ and $T$. Evidently there is at least one complex number $\lambda_0$ such that
$$
\|A-\lambda_0T\|=\inf_{\lambda\in \mathbb C} \|A-\lambda T\|.
$$
The number $\lambda_0$ is unique if $0\notin \sigma_{ap}(T)$ (or equivalently if $\inf\{\|Tx\|: \|x\|=1\}>0$). The proofs of existence and unicity of such $\lambda_0$ are detailed in \cite{paul}. Different authors, following \cite{stampfli}, called to this unique number as center of mass of $A$ respect to $T$ and we denote by $c(A, T)$ and when $T=I$ we write $c(A)$ and 
$D_A=\|A-c(A)I\|=\inf_{\lambda\in \mathbb C} \|A-\lambda I\|.$

Similarly for a norm $N(\cdot)$ we define
$$
D_{N, A}=\inf_{\lambda\in \mathbb C} N(A-\lambda I).
$$
With a similar proof as Theorem 1 in \cite{paul} we can state that $D_{N, A}=N(A-\lambda_0 I)$ for some $\lambda_0\in \mathbb{C}$.
If $A, X$ are two bounded linear operators on $\cH$, then in \cite{Abu-kittaneh2} the authors proved 
\begin{eqnarray}\label{ref_omar_kittaneh}
w(AX)\leq (\|A\|+D_A)w(X).
\end{eqnarray}
Since $D_A\leq \|A\|$, \eqref{ref_omar_kittaneh} is a refinement of the classical inequality $w(AX)\leq 2\|A\|w(X)$.

The following result is similar to \eqref{ref_omar_kittaneh} for the generalized numerical radius $\wn(\cdot)$.  Here we give a proof for the convenience of the reader. 

\begin{theorem} For any $A, X\in \bh$ and $N(\cdot)$ a selfadjoint and submultiplicative norm on $\bh$ hold
\begin{eqnarray}
w_N(AX)&\leq&\min \left\lbrace \frac12 \wn(AX\mp XA^*)+\frac12 \wn(AX\pm XA^*)\right\rbrace  \nonumber\\
&\leq&\min \left\lbrace N(A)\wn(X)+\frac12 \wn(AX\pm XA^*)\right\rbrace  \nonumber\\&\leq& (N(A)+D_{N, A}) \wn(X)\leq 2N(A)\wn(X)\leq 4\wn(A)\wn(X).
\end{eqnarray}
\end{theorem}
\begin{proof}
To prove the first and second inequality we observe that $AX=\frac 12(AX\pm XA^*)+\frac 12(AX\mp XA^*)$ and we use the inequality (4) in Lemma \ref{basic prop}.
Now, we consider $\lambda_0$ such that $N(A-\lambda_0 I)=D_{N, A}$ and we write $\lambda_0=e^{i\theta_0}|\lambda_0|$, then by the previous inequality we have 
\begin{eqnarray}
w_N(AX)&=&w_N(e^{-i\theta_0}AX)\leq N(e^{-i\theta_0}A) w_N(X)+\frac12 w_N(e^{-i\theta_0}AX-e^{i\theta_0}XA^*)\nonumber \\
&=&N(e^{-i\theta_0}A) w_N(X)+\frac12 w_N(e^{-i\theta_0}(A-\lambda_0I)X-e^{i\theta_0}X(A-\lambda_0 I)^*)\nonumber\\
&\leq&N(A) w_N(X)+ N(e^{-i\theta_0}(A-\lambda_0 I))w_N(X)=(N(A)+D_{N, A})w_N(X).\nonumber \
\end{eqnarray}
The fourth and fifth inequalities follows from the well-known facts: $D_{N, A}\leq N(A)$ and $N(A)\leq 2\wn(A)$.
\end{proof}

	\begin{corollary}
		Let $A, X\in \bh$, $X\neq 0$  and $N(\cdot)$ a selfadjoint and submultiplicative norm  on $\bh$. If $\wn(AX)=2N(A)\wn(X)$ or $\wn(AX)=4\wn(A)\wn(X)$then $A\perp_N  I$.
\end{corollary}

In order to obtain more  refinements for the case $w_N=w_p$, we need the following results obtained by Bhatia and Zhan in \cite{BZ1, BZ2}.
\begin{lemma}\label{bhatiazhan}
Let $T = A + iB \in \mathcal{B}_p(H)$  with $2\leq p<\infty$. If
$A $ is positive and $B$ is  selfadjoint then
\begin{equation}
\|T\|_p^2\leq \|A\|_p^2 + 2^{1-2/p} \|B\|_p^2.
\end{equation}
Also when both $A$ and $B$ are positive then 
\begin{equation}
\|T\|_p^2\leq \|A\|_p^2 +  \|B\|_p^2.
\end{equation}
\end{lemma}
Recall that an operator $A\in \bh$ is called accretive if $Re(A)\geq 0$ and analogously $A$ is dissipative if $Im(A)\geq 0$. Using these definitions, we have the following statement.
\begin{theorem}
Let $A\in\bh$,  $X\in \mathcal{B}_p(H)$ with  $2\leq p<\infty$  and $X$ accretive, then
\begin{equation}
w_p(AX)\leq  \sqrt{1+2^{1-2/p}} \|A\|  w_p(X).
\end{equation}
Moreover, if $X$ accretive and dissipative then 
\begin{equation}
w_p(AX)\leq \sqrt{2}\|A\| w_p(X).
\end{equation}
\end{theorem}
\begin{proof}
Since $\|\cdot\|_p$ is a selfadjoint norm, we recall that $w_p(AX)\leq \|AX\|_p \leq \|A\| \|X\|_p$.
Now, if $X$ is accretive then by Lemma \ref{bhatiazhan}, we have 
\begin{equation}
w_p(AX)\leq \|A\| \sqrt{\|Re(X)\|_p^2+2^{1-2/p}\|Im(X)\|_p^2}\leq \sqrt{1+2^{1-2/p}} \|A\| w_p(X).
\end{equation}
Furthermore, if $X$ is also dissipative then 
\begin{equation}
w_p(AX)\leq \|A\| \sqrt{\|Re(X)\|_p^2+\|Im(X)\|_p^2}\leq \sqrt{2} \|A\| w_p(X).
\end{equation}
\end{proof}

The following lemma plays a central role our following statements.
\begin{lemma}[\cite{kittaneh_jfa_1997}]Let $A,X\in \mathcal{B}(H)$ such that $AX$ is selfadjoint. If $XA$ belongs to the norm ideal associated with a unitarily invariant
norm $|||.|||$,  then $AX$ belongs to this ideal and 
$$|||AX|||\leq |||Re(XA)|||.$$
\end{lemma}
An important special case of the previous result asserts that if $AX$ is selfadjoint and $AX\in \mathcal{B}_p(H)$ then
\begin{equation}\label{cota NUI}
w_p(AX)=\|AX\|_p\leq\|Re(XA)\|_p \leq w_p(XA),
\end{equation}
for $1\leq p <\infty$.

\begin{theorem}
Let $A,X\in \mathcal{B}(H)$ such that $AX$ is selfadjoint and $AX\in \mathcal{B}_p(H)$. If $XA\in \mathcal{B}_p(H)$ with $1<p<\infty$, then the following conditions are equivalent
\begin{enumerate}
\item $w_p(AX)=w_p(XA).$
\item $\|Re(XA)\|_p=\|AX\|_p.$
\item $XA$ is selfadjoint.
\end{enumerate} 
\end{theorem}
\begin{proof}
{\rm (1) $\Rightarrow$ (2)} The equality follows from \eqref{cota NUI} and the hypothesis.

{\rm (2) $\Leftrightarrow$ (3)} The proof follows a similar idea to that of Kittaneh in \cite{kittaneh1992} Lemma 2.

{\rm (3) $\Rightarrow$ (1)} We have the following inequality 
$$
w_p(AX)=\|AX\|_p=\|Re(AX)\|_p\geq \|XA\|_p=w_p(XA).
$$
Combining with  inequality \eqref{cota NUI}, we obtain the equality desired. 
\end{proof}

Using Proposition 1 in \cite{kittaneh_jfa_1997} we obtain another  lower bound for the numerical radius for the product of two positive operators. Then if  $A, X \in \mathcal{B}(H)$
$$
\|X^{1/2}AX^{1/2}\|=\|A^{1/2}X^{1/2}\|^2 \leq \|Re(AX)\|\leq w(AX)
$$ 
and
$$
\|X^{1/2}AX^{1/2}\|_{p/2}=\|A^{1/2}X^{1/2}\|_p ^2\leq \|Re(AX)\|_{p/2}\leq w_{p/2}(AX)
$$
if $A, X\in \mathcal{B}_p(H)$ for $p\geq 2$.

\bibliographystyle{amsplain}.

\end{document}